\documentclass[12pt]{amsart}
\usepackage[utf8]{inputenc} 
\usepackage{amssymb}
\usepackage{amsfonts}
\usepackage{amsmath}
\usepackage{array}
\usepackage{enumitem}
\usepackage{tikz}
\usepackage[hyperindex=true,frenchlinks=true,colorlinks=true,
  linktocpage,
]{hyperref}

\newtheorem{example}{Example}[section]
\newtheorem{remark}[example]{Remark}

\newtheorem{theorem}[example]{Theorem}
\newtheorem{corollary}[example]{Corollary}

\newtheorem{definition}[example]{Definition}
\newtheorem{proposition}[example]{Proposition}
\newtheorem{algorithm}[example]{Algorithm}
\newtheorem{lemma}[example]{Lemma}

\def\SF{sub\-excedent function}

\def\maj{{\rm maj}}
\def\inv{{\rm inv}} 
\def\Des{\operatorname{Des}}

\def\tw{{\rm tw}}
\def\SG{{\mathfrak S}}
\def\SFE{\mathfrak{SF}}
\def\ie{\emph{i.e.}}
\def\goth{\mathfrak}

\def\CDes{\operatorname{DC}}
\def\GDes{\operatorname{GDes}}
\def\GC{\operatorname{GC}}

\def\Rec{\operatorname{Rec}}
\def\LC{\operatorname{LC}}
\def\tto{{\rm tot}}

\def\NCSF{{\bf Sym}}
\def\312{31\!-\!2}
\def\BST{\operatorname{BST}}

\definecolor{light-gray}{gray}{0.75}

\author{Arthur Nunge}
\title{An equivalence of multistatistics on permutations}
\keywords{Bijections, statistics, permutations, Catalan objects}

\begin{document}
\maketitle
\begin{abstract}
  We prove a conjecture of J.-C. Novelli,
  J.-Y. Thibon, and L. K. Williams (2010) about an equivalence of two
  triples of statistics on permutations. To prove this conjecture, we
  build a bijection through different combinatorial objects,
  starting with a Catalan-based object related to the PASEP.
  As a byproduct of this research, we also provide a new co-sylvester
  class-preserving bijection on permutations.
\end{abstract}

\section{Introduction}

The algebra of noncommutative symmetric functions \NCSF~\cite{NCSF1}
has been studied in algebraic combinatorics during the past twenty
years. It is a graded algebra
$\NCSF=\displaystyle\bigoplus_{n\geq 0}\NCSF_n$, where $\NCSF_n$ is of
dimension~$2^{n-1}$ for $n\geq 1$, so that its bases are indexed by
integer compositions, that is, finite sequences of positive integers
of sum~$n$. There are many purely combinatorial problems related to
this algebra such as, for example, the explicit description of the
relations between different bases through their transition
matrices. In this paper, we shall be interested in the basis
introduced by Tevlin in~\cite{Tev}, the so-called monomial basis of
\NCSF, for which the transition matrices~${\goth M}^{(n)}$ with the
ribbon basis have been described in~\cite{HNTT}.

In this last paper, the authors prove that the entry
${\goth M}^{(n)}_{I,J}$, indexed by two integer compositions~$I$
and~$J$, is equal to the number 
of permutations satisfying $\GC(\sigma) = I$ and $\Rec(\sigma) = J$,
where $\GC$ and $\Rec$ are two statistics that will be recalled
later. Some properties of this basis correspond to 
properties of the PASEP (Partially Asymmetric Simple Exclusion
Process), a physical model in which particles hop back and forth (and
in and out) of a one-dimensional lattice. More
precisely, in the study of the basis of Tevlin, the sum of the
entries of row $I$ of the transition matrix ${\goth M}^{(n)}$
corresponds to
the unnormalized steady-state probability of the state of the PASEP
with $n-1$ sites corresponding to $I$.

In the combinatorial version of the PASEP, the steady-state
probabilities are computed through the enumeration of the so-called
permutation tableaux, an object introduced in~\cite{CW}. These objects
are certain fillings of Ferrers diagrams. There
is a simple bijection between a state of the PASEP with~$N$ sites and
a Ferrers diagram of semi-perimeter~$N$ such that the steady-state
probability of this state of the PASEP corresponds to the sum of the
fillings of the Ferrers diagram as permutation tableaux. Moreover,
in~\cite{SW} the authors present a bijection between permutation
tableaux and permutations such that the permutation tableaux of a
given shape are sent to the permutations of a given~$\GC$ statistic.

In the PASEP context, there exists a natural $q$-statistic on
permutation tableaux that becomes the number of $\312$ patterns on
permutations (denoted by $\tto$) thanks to the bijection
of~\cite{SW}. It is then natural to define a $q$-analog of the
basis of Tevlin as the functions whose transition matrices
${\goth M}^{(n)}(q)$ with the ribbon basis are such that
${\goth M}^{(n)}_{I,J}(q)$ is the sum of the $q^{\tto(\sigma)}$ for
all~$\sigma$ satisfying $\GC(\sigma) = I$ and $\Rec(\sigma) = J$.

In~\cite{NTW}, the authors studied those matrices in an algebraic
way. However, the statistics $\GC$ and $\tto$ on permutations were not
appropriate for an algebraic study. So they built
suitable matrices~$\widetilde{{\goth M}}^{(n)}(q)$ for their algebraic
purpose through other statistics on permutations also recalled later
($\LC$, $\Rec$, and $\alpha$). Then, the
entry~$\widetilde{{\goth M}}^{(n)}_{I,J}(q)$ is the sum of the
$q^{\alpha(\sigma)}$ for all~$\sigma$ satisfying $\LC(\sigma) = I$ and
$\Rec(\sigma) = J$. They conjectured that their matrices are the same
as the~${\goth M}^{(n)}(q)$, or equivalently that the triples of
statistics $(\GC, \Rec, \tto)$ and $(\LC, \Rec, \alpha)$ are
equidistributed.

The aim of this paper is to prove this conjecture bijectively. As the
triple of statistics $(\LC, \Rec, \alpha)$ has a natural description
on \SF s, in order to gain in readability, we build a bijection
from permutations to \SF s sending the triple of statistics $(\GC,
\Rec, \tto)$ to $(\LC, \CDes, \alpha)$ where $\CDes$ is also defined
later. The global bijection is described as
a sequence of bijections through different combinatorial objects:
\begin{equation}
  \label{eqSchemaBijs_LH}
  \text{P} \overset{\psi_{FV}}{\longleftrightarrow} \text{LH}
  \overset{\psi}{\longleftrightarrow} \text{LLH}
  \overset{\phi_{2}}{\longleftrightarrow} \text{DWSF}
  \overset{\phi_{1}}{\longleftrightarrow} \text{SF},
\end{equation}
where P, LH, LLH, DWSF, and SF are respectively Permutations, Laguerre
Histories, Large Laguerre Histories, Decreasing Weighted Sub\-excedent
Functions, and Subexcedent Functions which are all recalled or defined
below. The main idea is to send permutations and \SF s to weighted Catalan objects
such that the weight corresponds to the third statistics in the
triples. To do that we use the well-known Fran\c con-Viennot
bijection $\psi_{FV}$ on one side, and we define a new bijection
$\phi_1$ building a weighted Catalan object from a subexcedent
function. Laguerre histories and large Laguerre histories are both
weighted Motzkin paths with similar conditions on the weights. In
order to avoid confusion and to simplify the description of the
map~$\psi$ we shall use another common representation on those objects
in terms of Weighted Dyck Paths (WDP). Diagram~\eqref{eqSchemaBijs_LH}
becomes:
\begin{equation}
  \label{eqSchemaBijs}
  \text{P} \overset{\psi_{FV}}{\underset{\ref{subSectFV}}\longleftrightarrow} \text{WDP}
  \overset{\psi}{\underset{\ref{subSectInvol}}\longleftrightarrow} \text{WDP}
  \overset{\phi_{2}}{\underset{\ref{subSectPhi2}}\longleftrightarrow} \text{DWSF}
  \overset{\phi_{1}}{\underset{\ref{subSectPhi1}}\longleftrightarrow} \text{SF}.
\end{equation}
Under each bijection in the diagram we indicate in which subsection
of the paper it is presented. 

In the last part of the paper we use the connection between large
Laguerre histories and permutations through a variation of $\psi_{FV}$
that we call $\psi_{FV}^0$
to define another triple of statistics close to $(\GC, \Rec, \tto)$
also giving a combinatorial interpretation for the entries of
${\goth M}^{(n)}(q)$. By combining $\psi_{FV}^0$ with $\psi_{FV}$ and
$\psi$ through the following diagram
\begin{equation}
  \text{P} \overset{\psi_{FV}}{\longleftrightarrow} \text{WDP}
  \overset{\psi}{\longleftrightarrow} \text{WDP}
  \overset{\psi_{FV}^0}{\longleftrightarrow}
  \text{P},
\end{equation}
we define a new bijection on permutations preserving the
co-sylvester classes, classes inherited from a monoid structure related
to the combinatorics of binary search trees~\cite{PBT}.

In Section~\ref{sectionBackground} we give the background and
notations needed in the rest of the paper.
In Section~\ref{sectionPermToDyckPath} we describe the
bijections $\psi_{FV}$ and $\psi$ which are natural bijections leading
to objects
where the statistics $\GC$ and $\tto$ are naturally defined. The
object obtained is a pair consisting of a Catalan object and a
weight. In Section~\ref{sectionSFToDyckPath} we begin by building
a Catalan object associated with a weight from \SF s through $\phi_1$
and then describe $\phi_2$, a Catalan bijection between decreasing \SF
s and Dyck paths.  In Section~\ref{sectionGlobal} we prove the
conjectures of~\cite{NTW} and give some properties associated with the
global bijection. Finally, in Section~\ref{sectionOtherConv} we
introduce a new co-sylvester class-preserving bijection and a new
triple of statistics.

\section{Notations and background}
\label{sectionBackground}
\subsection{Permutations, compositions, and \SF s}

Let us first fix our notations concerning permutations. We represent a
permutation $\sigma$ as a word $\sigma_1\sigma_2\cdots\sigma_n$ such
that $\sigma_i = \sigma(i)$. For all the forthcoming examples we fix
$\tau=528713649$. We shall sometimes use the notation $[n]=[1,n]$ to
denote the set $\{1,2,\cdots,n\}$.

A {\it recoil} of a permutation $\sigma$ is a value $i$ such that
$i+1$ is on the left of $i$ or equivalently such that
$\sigma^{-1}_i>\sigma^{-1}_{i+1}$.  The recoil set of $\sigma$ is the set of
the values of recoils.  For example, the recoil set of $\tau$ is
$\{ 1, 4, 6, 7\}$. A $\312$ pattern of $\sigma$ is a pair $(i,j)$
such that $j > i+1$ and $\sigma_{i+1} < \sigma_j < \sigma_i$. We
denote by $\tto(\sigma)$ the number of $\312$ patterns of $\sigma$. For
example, $\tto(\tau) = 5$. We shall often need the
number of times that a value appears as a $2$ in a $\312$ pattern. We
define $\tto_k(\sigma)$ as the number of times $k$ appears as a $2$ in a
$\312$ pattern in $\sigma$.

A \emph{composition} of an integer $n$ is a sequence
$I=(i_1,\dots,i_r)$ of positive integers of sum $n$.  The integer $r$
is called the \emph{length} of the composition.  The \emph{descent
  set} of $I$ is $\Des(I) = \{ i_1,\ i_1+i_2, \ldots ,
i_1+\dots+i_{r-1}\}$.

The {\it major index} $\maj(I)$ of a composition is the sum of the
values in the descent set of $I$. For example, $\maj(1,3,2,1,2)=18$.

The \emph{recoil composition} $\Rec(\sigma)$ of a permutation
$\sigma\in\SG_n$ is the composition of $n$ whose descent set is the
recoil set of $\sigma$. For example, $\Rec(\tau)=(1,3,2,1,2)$.

The \emph{Genocchi descent set}~\cite{HNTT} of a permutation
$\sigma\in\SG_n$, denoted by $\GDes(\sigma)$, is the set of values immediately followed by a
smaller value (it is sometimes called the descent tops of
$\sigma$~\cite{SW}).
The \emph{Genocchi composition of descents} (or
G-compos\-ition, for short) $\GC(\sigma)$ of a permutation is the
integer composition $I$ of $n$ whose descent set is
$\{d-1~|~d\in\GDes(\sigma)\}$. For example,
we have $\GDes(\tau) = \{5,6,7,8\}$ and $\GC(\tau)=(4,1,1,1,2)$.

About the statistic $\LC$, the definition on permutations given
in~\cite{NTW} is an algorithm and starts by taking the Lehmer code of
the inverse permutation, a bijection between permutations and
{\SF}s. We shall directly define it on {\SF}s as the other two
statistics of the triple $(\LC, \Rec, \alpha)$ are easily defined on
this object. 

A \emph{\SF} of size $n$ is a word $u$ of size $n$ on the alphabet of
nonnegative integers such that for each $i \in \{ 1, \ldots, n\}$ we
have $u_i \leq n-i$. We denote by $ \SFE_n$ the set of \SF s of
size~$n$.  They are enumerated by $n!$ as they are in bijection with
permutations through the Lehmer code.  The \emph{statistic $LC$} is
defined on a {\SF} $u$ as follows.
\begin{itemize}
\item Set $S=\emptyset$ and read $u$ from right to left.  At each
  step, if the entry $k$ is strictly greater than the size of $S$, add
  the $(k-|S|)$-th element of the sequence $[1,n]$ not considering the
  elements of $S$.
\item The set $S$ is the descent set of a composition $C$, and
  $\LC(u)$ is the mirror image $\overline{C}$ of $C$.
\end{itemize}
For example, with $u= 315503200$, $S$ is $\emptyset$ at first, then
the set $\{2\}$ at the third step as the third letter from the right
is a $2$ and $S$ was empty, then $\{2,3\}$ (fourth step), then
$\{2,3,5\}$ (sixth step), then $\{2,3,4,5\}$ (seventh step).  Hence
$C$ is $(2,1,1,1,4)$, so that $\LC(u)=(4,1,1,1,2)$.

Define the
\emph{number of inversions} of a {\SF} as the sum of its values. We
also define \emph{the descent set} of a {\SF} as
\begin{equation}
  \Des(u) = \{ i\in \{1, \ldots, n-1\} ~|~ u_i > u_{i+1}\}
\end{equation}
and the \emph{descent composition} of $u$ (denoted by $\CDes(u)$) as
the composition of size $n$ whose descent set is $\Des(u)$. One easily
checks that $\CDes(315503200) = (1,3,2,1,2)$. This definition of the
descent composition of a {\SF} corresponds to the recoil composition of
the permutation before taking the Lehmer code of its inverse.

\subsection{Different weighted paths}
\label{sectionPaths}
Recall that a Dyck path of size $n$ is a path with $n$ increasing
steps and $n$ decreasing steps that never goes below the horizontal
axis. A Motzkin path of size $n$ is a path with $n$ steps among
increasing steps, decreasing steps, and horizontal steps that never
goes below the horizontal axis. For any path $P$ we denote by~$P_i$
the $i$-th step of this path. We call the height of a step the
distance between the beginning of this step and the horizontal axis.
A weight associated with a path is a nonnegative integer vector with the
same size as the size of the path.

\begin{definition}
  A Laguerre history $L$ of size $n$ is a weighted Motzkin path of size
  $n$ with two kinds of horizontal steps ($\longrightarrow$ and
  $\dashrightarrow$). The weight $w$ is a sequence of integers of size
  $n$ satisfying for all $i \leq n$,
  \begin{itemize}
  \item $0 \leq w_i \leq h_i$ if $L_i$ is $\nearrow$ or
    $\longrightarrow$;
  \item $0 \leq w_i \leq h_i-1$ if $L_i$ is $\searrow$ or
    $\dashrightarrow$;
  \end{itemize}
  where $h_i$ is the height of the $i$-th step of $L$.

  A large Laguerre history of size $n-1$ is a weighted Motzkin path of
  size $n-1$ with two kinds of horizontal steps where for any
  $i\leq n-1$, the weight $w$ satisfies $0 \leq w_i \leq h_i$.
\end{definition}

Figure~\ref{exLH} represents a Laguerre history of size $9$ and a
large Laguerre history of size $8$. For a better readability we only
wrote the strictly positive~$w_i$.

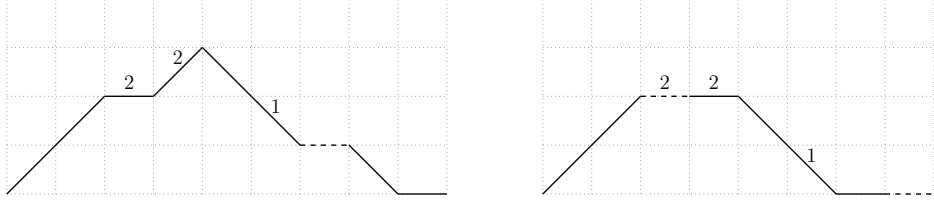
\begin{figure}[h!t]
  \begin{center}
    \begin{tikzpicture}
      \node (F1) at (0,0) {
        \scalebox{0.65}{
          \begin{tikzpicture}
  \draw [thin, gray, dotted] (0, 0) grid (9, 4);
  \draw [thick] (0, 0) -- (1, 1) node[midway,scale=1,above]{};
  \draw [thick] (1, 1) -- (2, 2) node[midway,scale=1,above]{};
  \draw [thick] (2, 2) -- (3, 2) node[midway,scale=1,above]{2};
  \draw [thick] (3, 2) -- (4, 3) node[midway,scale=1,above]{2};
  \draw [thick] (4, 3) -- (5, 2) node[midway,scale=1,above]{};
  \draw [thick] (5, 2) -- (6, 1) node[midway,scale=1,above]{1};
  \draw [thick, dashed] (6, 1) -- (7, 1) node[midway,scale=1,above]{};
  \draw [thick] (7, 1) -- (8, 0) node[midway,scale=1,above]{};
  \draw [thick] (8, 0) -- (9, 0) node[midway,scale=1,above]{};
\end{tikzpicture}
         }
      };
      \node (F2) at (6.8,0) {
        \scalebox{0.65}{
          \begin{tikzpicture}
  \draw [thin, gray, dotted] (0, 0) grid (8, 4);
  \draw [thick] (0, 0) -- (1, 1) node[midway,scale=1,above]{};
  \draw [thick] (1, 1) -- (2, 2) node[midway,scale=1,above]{};
  \draw [thick, dashed] (2, 2) -- (3, 2) node[midway,scale=1,above]{2};
  \draw [thick] (3, 2) -- (4, 2) node[midway,scale=1,above]{2};
  \draw [thick] (4, 2) -- (5, 1) node[midway,scale=1,above]{};
  \draw [thick] (5, 1) -- (6, 0) node[midway,scale=1,above]{1};
  \draw [thick] (6, 0) -- (7, 0) node[midway,scale=1,above]{};
  \draw [thick, dashed] (7, 0) -- (8, 0) node[midway,scale=1,above]{};
\end{tikzpicture}
         }
      };
    \end{tikzpicture}
    \caption{A Laguerre history of size 9 and a large Laguerre history
      of size 8.}
    \label{exLH}
  \end{center}
\end{figure}

Both objects are in bijection and enumerated by $n!$. A common way to
build a bijection between these is to represent each by a
weighted Dyck path.

\begin{definition}
  A weighted Dyck path of size $n$ is a Dyck path
  associated with a weight both of size $n$. For all $i$,
  the weight satisfies $w_i \leq (h_i-1)/2$ where $h_i$ is the height
  of the $2i$-th step of the Dyck path.
\end{definition}

An example of a weighted Dyck path of size $9$ is given in
Figure~\ref{figHist}. For a better
readability we grouped the steps two by two between vertical dashed
lines and only wrote the strictly positive $w_i$ above steps $D_{2i-1}$
and $D_{2i}$. In this representation, $h_i$ is exactly the height of
the meeting point of the $(2i-1)$-st and the $2i$-th steps of the Dyck
path.

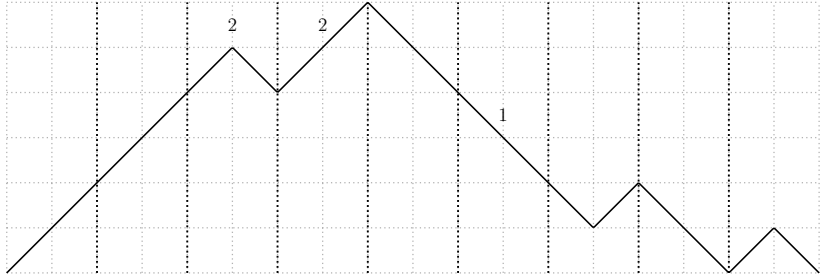
\begin{figure}[h!t]
  \begin{center}
    \scalebox{0.6}{
      \begin{tikzpicture}
  \draw[dotted, thick, color=gray!60] (0, 0) grid (18, 6);
  \draw[rounded corners=1, color=black, line width=1] (0, 0) -- (1, 1)
  -- (2, 2) -- (3, 3) -- (4, 4) -- (5, 5) -- (6, 4) -- (7, 5) -- (8,
  6) -- (9, 5) -- (10, 4) -- (11, 3) -- (12, 2) -- (13, 1) -- (14, 2)
  -- (15, 1) -- (16, 0) -- (17, 1) -- (18, 0);

  \draw[dotted, very thick] (2, 0) -- (2, 6); \draw[dotted, very
    thick] (4, 0) -- (4, 6); \draw[dotted, very thick] (6, 0) -- (6,
  6); \draw[dotted, very thick] (8, 0) -- (8, 6); \draw[dotted, very
    thick] (10, 0) -- (10, 6); \draw[dotted, very thick] (12, 0) --
  (12, 6); \draw[dotted, very thick] (14, 0) -- (14, 6); \draw[dotted,
    very thick] (16, 0) -- (16, 6);
  
  \node (pond1) at (5, 5.5) {2}; \node (pond2) at (7, 5.5) {2}; \node
  (pond2) at (11, 3.5) {1};
\end{tikzpicture}
 }
    \caption{A weighted Dyck path of size 9.}
    \label{figHist}
  \end{center}
\end{figure}

The map sending a Laguerre history to a weighted Dyck path is
described by the following algorithm.

\begin{algorithm}
  \label{algoLH}
  Let $L$ be a Laguerre history of size $n$ and $i \in [n]$. Then
  the weighted Dyck path $D$ associated with $L$ satisfies
  \begin{itemize}
  \item $D_{2i-1} = D_{2i} = /$ if $L_i = \nearrow$,
  \item $D_{2i-1}D_{2i} = /\backslash$ if $L_i = \longrightarrow$,
  \item $D_{2i-1}D_{2i} = \backslash/$ if $L_i = \dashrightarrow$,
  \item $D_{2i-1} = D_{2i} = \backslash$ if $L_i = \searrow$.
  \end{itemize}
  The weight remains the same.
\end{algorithm}

To build a weighted Dyck path from a large Laguerre history $L$,
start by applying Algorithm~\ref{algoLH} to $L$ then add an extra $/$
step at the beginning and $\backslash$ step at the end. The weight
also remains the same, we just add a $0$ at its end to obtain a weight
of size $n$.

Applying the corresponding algorithm to both paths of
Figure~\ref{exLH} gives the weighted Dyck path of
Figure~\ref{figHist}.

The fact that both objects are enumerated by $n!$ comes from the
Fran\c con-Viennot bijection.

\subsection{The Fran\c con-Viennot bijection}

The Fran\c con-Viennot bijection, first described in~\cite{FV}, is a
bijection between permutations and Laguerre histories. We shall
describe a bijection between permutations and weighted Dyck paths
(defined below) that can be obtained by applying
Algorithm~\ref{algoLH} to the result of the version  of Corteel of
the Fran\c con-Viennot bijection described in~\cite{Cor}.

This bijection corresponds to the first part of
Diagram~\eqref{eqSchemaBijs}.

\begin{equation*}
  \text{P} \overset{\psi_{FV}}{\longleftrightarrow} \text{WDP}
  \textcolor{light-gray}{\overset{\psi}{\longleftrightarrow} \text{WDP}
  \overset{\phi_{2}}{\longleftrightarrow} \text{DWSF}
  \overset{\phi_{1}}{\longleftrightarrow} \text{SF}}
\end{equation*}

In order to compute each step of the Dyck path, the
Fran\c con-Viennot bijection compares the values of $\sigma$ with both
of their neighbors. In order to do so with both the first and the last
values of $\sigma$ we use the convention $\sigma_0 = 0$ and
$\sigma_{n+1} = n+1$.

\begin{algorithm}[Fran\c con-Viennot]
  \label{algoFV}
  Let $\sigma \in \SG_n$, $j \in \{1, \ldots, n\}$ and $k = \sigma_j$.
  Then in the Dyck path of $\psi_{FV}(\sigma)$ we have
  \begin{itemize}
  \item $D_{2k-1} = /$ if $\sigma_j < \sigma_{j+1}$ and $D_{2k-1} =
    \backslash$ otherwise,
  \item $D_{2k} = \backslash$ if $\sigma_{j-1} < \sigma_{j}$ and
    $D_{2k} = /$ otherwise.
  \end{itemize}
  The weight is built as follows: $w_k$ is equal to
  $\tto_k(\sigma)$, the number of $\312$ patterns where $k$ stands for
  the~$2$.
\end{algorithm}

For example, if we consider $\sigma = 528713649$, its image by
$\psi_{FV}$ is the weighted Dyck path of Figure~\ref{figHist}. One can
check that the values $1$, $2$, $4$ and $7$ are smaller than their
left neighbor in $\sigma$ and values 1, 2, 4, and 9 are smaller than
their right neighbor. The $\312$ patterns are 52-3, 71-3, 52-4, 71-4,
and 71-6 so that the weight is indeed $(0,0,2,2,0,1,0,0,0)$.

Let us also describe the reverse algorithm building a permutation
from a weighted Dyck path $(D, w)$.

\begin{algorithm}
  \label{algoInverseFV}
  The permutation $\sigma$ is built iteratively by
  \begin{itemize}
  \item Initialization: $\sigma = \circ$;
  \item At the $k$-th step of the algorithm, replace the $(w_k+1)$-st $
    \circ$ of $\sigma$ by:
    \begin{itemize}
    \item $\circ k \circ$ if $D_{2k-1} = D_{2k} = /$,
    \item $k \circ$ if $D_{2k-1}D_{2k} = /\backslash$,
    \item $\circ k$ if $D_{2k-1}D_{2k} = \backslash/$,
    \item $ k$ if $D_{2k-1} = D_{2k} = \backslash$;
    \end{itemize}
  \item The final permutation is obtained by removing the last $\circ$.
  \end{itemize}
\end{algorithm}

For the weighted Dyck path of Figure~\ref{figHist}, the previous
algorithm gives:
$$
\begin{array}{c}
  \sigma = \circ ~\to~ \circ 1 \circ ~\to~ \circ 2 \circ 1 \circ ~\to~
  \circ 2 \circ 13 \circ ~\to~
  \circ2\circ13\circ4\circ\vspace{0.2cm}\\
  \to~  52\circ13\circ4\circ ~\to~ 52 \circ 1364 \circ
  ~\to~ 52 \circ 71364 \circ \vspace{0.2cm}\\~\to~ 52871364 \circ ~\to~
  528713649\circ ~\to~ 528713649.
\end{array}
$$

\section{Permutations to weighted Dyck paths}
\label{sectionPermToDyckPath}

Let us describe what happens to the triple of statistics on
permutations through the Fran\c con-Viennot bijection.

\subsection{The statistics through the Fran\c con-Viennot bijection}
\label{subSectFV}
\begin{definition}
  Let $(D,w)$ be a weighted Dyck path of size $n$.
  \begin{itemize}
  \item The \emph{total weight} of $(D,w)$ (denoted by $\tw(D,w)$) is
    the sum of the values of $w$.
  \item The \emph{descent set} of $(D,w)$ is
    \begin{equation}
      \Des(D,w) = \{i~|~w_i > w_{i+1}\} \cup \{ i~|~w_i = w_{i+1}, D_{2i}
      = / \}
    \end{equation}
    and the \emph{descent composition} $\CDes(D,w)$ is the composition
    of $n$ whose descent set is $\Des(D,w)$.
  \item The \emph{Genocchi descent set} of $(D,w)$ is
    \begin{equation}
      \GDes(D,w) = \{ i \in [2, n]~|~D_{2i-1} = \backslash\}
    \end{equation}
    then, as for permutations, the \emph{Genocchi composition of
      descent} of $(D,w)$ (denoted by $\GC(D,w)$) is the composition of $n$
    whose descent set is $\{d-1~|~d \in \GDes(D,w)\}$.
  \end{itemize}
\end{definition}

On the weighted Dyck path $(D,w)$ of Figure~\ref{figHist} we have
$\tw(D,w) = 5$. The positions $i$ such that $D_{2i-1} = \backslash$
are $\GDes(D,w) = \{5, 6, 7, 8\}$ so that $\GC(D,w) = (4,1,1,1,2)$ and
one can check that $\CDes(D,w) = (1,3,2,1,2)$.

\begin{proposition}
  \label{lemmaStatsFV}
  Let $\sigma$ be a permutation of size $n$ and $(D,w) =
  \psi_{FV}(\sigma)$. We have the following properties:
  \begin{itemize}
  \item $\tw(D,w) = \tto(\sigma)$;
  \item $\GC(D,w) = \GC(\sigma)$;
  \item $\CDes(D,w) = \Rec(\sigma)$.
  \end{itemize}
\end{proposition}

\begin{proof}
 The first two assertions come directly from our description of
 $\psi_{FV}$. To prove the last one, it is easier to work with 
 $\psi_{FV}^{-1}$ that builds a permutation $\sigma$ from a weighted
 Dyck path $(D, w)$, as described in Algorithm~\ref{algoInverseFV}.  Let $k \in
 \{1, \ldots, n-1\}$. If $w_k < w_{k+1}$, then $k$ is on the left of
 $k+1$ in $\sigma$ so it is not a recoil for $\sigma$. Using the same
 idea, if $w_k > w_{k+1}$ then $k$ is a recoil for $\sigma$. Now, when
 $w_k = w_{k+1}$, we have a recoil in $k$ if and only if there is a
 new $\circ$ on the left of $k$ when we place it in $\sigma$ which
 happens if and only if $D_{2k} = /$.
\end{proof}

We now have a representation of the triple of statistics on the
weighted Dyck paths but the definition of $\GC$ on those objects is
not very natural. Indeed, the fact that $i$ belongs to $\GDes(D, w)$ or
not depends on the $(i+1)$-st group of steps. To simplify
the second part of the construction, we
transform the Dyck paths to obtain a more suitable statistic.

\subsection{An involution on Dyck paths}
\label{subSectInvol}

\begin{definition}
  \label{defInvol}
  Let $D$ be a Dyck path of size $n$. The Dyck path $\psi(D)$ is
  obtained by sending every pair of steps $/\backslash$ or
  $\backslash/$ at positions $2i$, $2i+1$ to one another.
\end{definition}

We extend this definition to weighted Dyck paths by carrying the
weight.  Note that it is an involution. An example is given in
Figure~\ref{figPsi} where we apply $\psi$ to the weighted Dyck path of
Figure~\ref{figHist}.

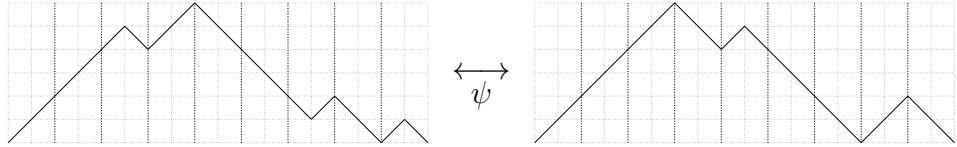
\begin{figure}[h!t]
  \begin{center}
    \begin{tikzpicture}
      \node (F1) at (0,0) {
        \scalebox{0.31}{
          \begin{tikzpicture}
  \draw[dotted, thick, color=gray!60] (0, 0) grid (18, 6);
  \draw[rounded corners=1, color=black, line width=1] (0, 0) -- (1, 1)
  -- (2, 2) -- (3, 3) -- (4, 4) -- (5, 5) -- (6, 4) -- (7, 5) -- (8,
  6) -- (9, 5) -- (10, 4) -- (11, 3) -- (12, 2) -- (13, 1) -- (14, 2)
  -- (15, 1) -- (16, 0) -- (17, 1) -- (18, 0);

  \draw[dotted, very thick] (2, 0) -- (2, 6); \draw[dotted, very
    thick] (4, 0) -- (4, 6); \draw[dotted, very thick] (6, 0) -- (6,
  6); \draw[dotted, very thick] (8, 0) -- (8, 6); \draw[dotted, very
    thick] (10, 0) -- (10, 6); \draw[dotted, very thick] (12, 0) --
  (12, 6); \draw[dotted, very thick] (14, 0) -- (14, 6); \draw[dotted,
    very thick] (16, 0) -- (16, 6);
\end{tikzpicture}
         }
      };

      \node (arrow) at (3.5,0) {$\longleftrightarrow$};
      \node (phi) at (3.5,-0.3) {$\psi$};

      \node (F1) at (7,0) {
        \scalebox{0.31}{
          \begin{tikzpicture}
  \draw[dotted, thick, color=gray!60] (0, 0) grid (18, 6);
  \draw[rounded corners=1, color=black, line width=1] (0, 0) -- (1, 1)
  -- (2, 2) -- (3, 3) -- (4, 4) -- (5, 5) -- (6, 6) -- (7, 5) -- (8,
  4) -- (9, 5) -- (10, 4) -- (11, 3) -- (12, 2) -- (13, 1) -- (14, 0)
  -- (15, 1) -- (16, 2) -- (17, 1) -- (18, 0);

  \draw[dotted, very thick] (2, 0) -- (2, 6); \draw[dotted, very
    thick] (4, 0) -- (4, 6); \draw[dotted, very thick] (6, 0) -- (6,
  6); \draw[dotted, very thick] (8, 0) -- (8, 6); \draw[dotted, very
    thick] (10, 0) -- (10, 6); \draw[dotted, very thick] (12, 0) --
  (12, 6); \draw[dotted, very thick] (14, 0) -- (14, 6); \draw[dotted,
    very thick] (16, 0) -- (16, 6);
\end{tikzpicture}
         }
      };
\end{tikzpicture}
\caption{An example of $\psi$.}
\label{figPsi}
\end{center}
\end{figure}

Definition~\ref{defInvol} is a visual way of defining $\psi$, a more
formal way would be by using the following lemma.

\begin{lemma}
  \label{otherDefPsi}
  Let $D$ be a Dyck path, we have:
  \begin{equation}
    \label{reformPsi}
    \begin{array}{rcl}
      D_{2i} & = & \psi(D)_{2i+1}; \\
      D_{2i+1} & = & \psi(D)_{2i}.
    \end{array}
  \end{equation}
\end{lemma}

\begin{proof}
  If $D_{2i} = D_{2i+1}$ then $\psi$ does not change those step so the
  result is clear. Otherwise $D_{2i} \ne D_{2i+1}$ and then $\psi$
  sends one to the other which ends the proof.
\end{proof}

This involution corresponds to the second part of
Diagram~\eqref{eqSchemaBijs}.

\begin{equation*}
  \textcolor{light-gray}{\text{P} \overset{\psi_{FV}}{\longleftrightarrow}}
  \text{WDP} \overset{\psi}{\longleftrightarrow} \text{WDP}
  \textcolor{light-gray}{\overset{\phi_{2}}{\longleftrightarrow} \text{DWSF}
    \overset{\phi_{1}}{\longleftrightarrow} \text{SF}}
\end{equation*}

This bijection gives a new bijection between Laguerre histories and
large Laguerre histories sending the equivalent of the $\GC$
statistic to a very similar one. Unfortunately this bijection is not
intuitive on Laguerre objects.

Let us introduce the new statistics on weighted Dyck paths
corresponding to $\CDes$ and $\GC$ after applying $\psi$.

\begin{definition}
  The statistics $\GC^0$ and $\CDes^0$ are defined as follows.
  \begin{itemize}
  \item The \emph{Genocchi descent set of type 0} of a weighted Dyck
    path of size $n$ is
    \begin{equation}
      \GDes^0(D,w) = \{ i \in [1, n-1]~|~D_{2i} = \backslash\}
    \end{equation}
    and the \emph{Genocchi composition of descent of type 0} (denoted by
    $\GC^0(D, w)$) is the
    composition of $n$ whose descent set is $\GDes^0(D,w)$.
  \item The \emph{descent set of type 0} of $(D,w)$ is
    \begin{equation}
      \Des^0(D,w) = \{i~|~w_i > w_{i+1}\} \cup \{ i~|~w_i = w_{i+1},
      D_{2i+1} = / \}
    \end{equation}
    and the \emph{descent composition of type 0} (denoted by $\CDes^0(D,w)$)
    is the composition of $n$ whose descent set is $\Des^0(D,w)$.
  \end{itemize}
\end{definition}

\begin{proposition}
\label{lemmaStatsInvol}
 Let $(D,w)$ be a weighted Dyck path, then
 \begin{itemize}
  \item $\tw(\psi(D,w)) = \tw(D,w)$;
  \item $\GC^0(\psi(D,w))= \GC(D,w)$;
  \item $\CDes^0(\psi(D,w)) = \CDes(D,w)$.
 \end{itemize}
\end{proposition}

\begin{proof}
  As $\psi$ does not change the weight of the Dyck path, the total
  weight is carried.
  
  For the two other points, the statistics are really close but just
  differ on the index of the considered steps, so
  Lemma~\ref{otherDefPsi} ends the proof.
\end{proof}

\section{Subexcedent functions to weighted Dyck paths}
\label{sectionSFToDyckPath}
The aim of this section is to build a bijection from \SF s to weighted
Dyck paths sending the triple $(\LC, \CDes,
\inv-\maj(\overline{\LC}))$ to $(\GC^0, \CDes^0, \tw)$ where the
statistic $\inv-\maj(\overline{\LC})$ corresponds to the statistic
$\alpha$ in the introduction. To do this, we build a bijection
from \SF s to an intermediate object: \emph{decreasing weighted \SF s}
which are represented by a Catalan object and a weight, and then
build a Catalan bijection between decreasing \SF s and Dyck paths.

\subsection{Subexcedent functions to decreasing weighted \SF s}
\label{subSectPhi1}

\begin{definition}
  \label{defDecreasingLehmer}
  Let us define a \emph{decreasing {\SF}} of size $n$ and a
  \emph{weight} for it.
  \begin{itemize}
  \item A {\SF} $u$ of size $n$ is \emph{decreasing} if the word
    obtained by removing all its zeroes is a strictly decreasing word.
  \item A \emph{weight} associated with a decreasing {\SF} is a word $w$ of size
    $n$ such that for all $k \in \{1, \ldots, n\}$, $w_k$ is smaller
    than or equal to the number of $i < k$ such that $0 < u_i \leq n-k$
    (\ie, the number of positive values on the left of $k$ that could
    be at position $k$).
  \end{itemize}
\end{definition}

For example, the {\SF} $u = 540300200$ is decreasing. As for the
associated weight, the maximum weight $W$ is the weight for which each
value is maximal and equal to the number of positive values to its
left smaller than $n$ minus its position. The maximum weight
of~$u$ is $012221000$ so the weight $002201000$ is correct since it is
smaller than the maximum weight component-wise, whereas $000002000$ is not.

The decreasing \SF s are indeed Catalan objects since one can build a
bijection with nondecreasing parking functions, the nondecreasing
words whose $i$-th value is a positive integer smaller or equal to $i$.
\begin{algorithm}
  Let $u$ be a decreasing \SF.
  \begin{itemize}
  \item Let $v$ be the mirror image of u,
  \item replace each 0 of $v$ by the first nonzero value to its left (if
    such a value exists),
  \item add one to each value of $v$.
  \end{itemize}
\end{algorithm}
For example, if $u=540300200$, the mirror image is $002003045$ and the
associated nondecreasing parking function is $113334456$.

\subsubsection{Description of the bijection $\phi_1$ between \SF s and decreasing weighted \SF s}

This new bijection $\phi_1$ can be described as an algorithm that
sorts a {\SF} by successively moving the greatest value to its left.

\begin{algorithm}
  \label{algoPhi1}
  Let $u$ be a {\SF} of size $n$. Set the weight $w$ to $0^n$.
  \begin{itemize}
  \item Step 1: define the \emph{pivot} as the greatest value in $u$
    such that one of its occurrences has smaller or equal nonzero
    values to its left. If the pivot is not defined, the algorithm
    stops. Otherwise, let $k$ be the position of the rightmost
    occurrence of the pivot in $u$.
  \item Step 2: among the values smaller than or equal to the pivot on
    its left, let $i$ be the position of the rightmost occurrence of
    the largest one. Modify the {\SF} by decrementing $u_i$ by $1$
    and then swapping $u_i$ with $u_k$. Modify the weight by
    incrementing $w_k$. Go back to Step 1.
  \end{itemize}
\end{algorithm}
Then $\phi_1(u)$ is the resulting pair $(u,w)$ of the algorithm.

Let us give an example with $u = 315503200$. Our algorithm follows the
steps:
\begin{enumerate}[label=\arabic*)]
\item $u = 315503200$, $w = 000000000$, then $pivot = 5$, $k = 4$ and
  $i = 3$;
\item $u = 31{\bf \textcolor{red}{54}}03200$, $w = 000{\bf
  \textcolor{red}{1}}00000$, then $pivot = 5$, $k = 3$ and $i = 1$;
\item $u = {\bf \textcolor{red}{5}}1{\bf \textcolor{red}{2}}403200$,
  $w = 00{\bf \textcolor{red}{1}}100000$, then $pivot = 4$, $k = 4$
  and $i = 3$;
\item $u = 51{\bf \textcolor{red}{41}}03200$, $w = 001{\bf
  \textcolor{red}{2}}00000$, then $pivot = 4$, $k = 3$ and $i = 2$;
\item $u = 5{\bf \textcolor{red}{40}}103200$, $w = 00{\bf
  \textcolor{red}{2}}200000$, then $pivot = 3$, $k = 6$ and $i = 4$;
\item $u = 540{\bf \textcolor{red}{3}}0{\bf \textcolor{red}{0}}200$,
  $w = 00220{\bf \textcolor{red}{1}}000$ and the algorithm stops.
\end{enumerate}
At the end, $\phi_1(315503200) = (540300200, 002201000)$.

This map corresponds to the last part of
Diagram~\eqref{eqSchemaBijs}.

\begin{equation*}
  \textcolor{light-gray}{\text{P} \overset{\psi_{FV}}{\longleftrightarrow}
  \text{WDP} \overset{\psi}{\longleftrightarrow} \text{WDP}
  \overset{\phi_{2}}{\longleftrightarrow}} \text{DWSF}
  \overset{\phi_{1}}{\longleftrightarrow} \text{SF}
\end{equation*}

\begin{proposition}
 The map $\phi_1$ is a well-defined function from \SF s to decreasing
 weighted \SF s.
\end{proposition}

\begin{proof}
 It is clear that $\phi_1$ is well-defined for every {\SF}~$u$ and
 that the algorithm stops (we have at most $n-1-pivot$
 swaps for each pivot).
 
 The result is decreasing in the sense of
 Definition~\ref{defDecreasingLehmer} because the algorithm sorts the
 \SF.
 
 To prove that the result is a decreasing weighted \SF, we need to
 prove that the weight satisfies the constraints of
 Definition~\ref{defDecreasingLehmer}. The value of the weight at
 position $j$ was increased at most once per pivot that ended on its
 left, so it is smaller than the number of non-zero values
 on the left of position $j$ at the end of the algorithm. Moreover, it was
 increased only if the pivot was at position $j$, which is possible
 only if the pivot is smaller than or equal to $n-j$. So $w$
 satisfies the constraints of the definition.
\end{proof}

\begin{lemma}
  \label{lemmaBij2And2}
  Consider an execution of Step 2 of
  Algorithm~\ref{algoPhi1}. Immediately after that we can recover both
  positions $i$ and $k$. 
  
  Let $j$ be the rightmost position such that $w_j \ne 0$. The
  position $i$ of the pivot is the rightmost position on the left of
  $j$ such that $u_i > u_j$.
  
  The previous position $k$ of the pivot is the nearest position to
  the right of $i$ such that the associated weight is nonzero, \ie,
  $k$ is the smallest $k > i$ such that $w_k \ne 0$.
\end{lemma}

\begin{proof}
  A value $u_i$ is \emph{weighted} if the corresponding weight is
  nonzero, \ie, $w_i > 0$. 
  
  We start by proving the second part of the lemma. Consider a
  step $S$ where a pivot $p$ in position $k$ should exchange with
  the value in position $i$. Then, for all $j$ such that $i <j <k$, we 
  have $w_j=0$. Indeed, assume there is a $j$ such that $i<j<k$ and
  $w_j \ne 0$ and let $S'$ be the step corresponding to the previous exchange
  concerning $u_j$ and a pivot $p'$. As $u_j$ is to the right of 
  $u_i$ in step $S$, we must have $p'$ to the right of both $u_i$
  and $u_j$ in step $S'$. Moreover, as $p'$ exchange with $u_j$,
  we have either $u_j > u_i$ and after the exchange and the
  decrementation this inequality becomes $u_j \geq u_i$, or
  $u_j=u_i$ which implies that $u_j$ is to the right of $u_i$ since the
  exchange implies $u_j$, and that $p'$ exchanges with $u_i$ next. In
  both cases we end up with $u_j \geq u_i$ and as there are no
  exchange implying $u_j$ between $S'$ and $S$, only $u_i$ may be
  decremented so we also have $u_j \geq u_i$ in situation $S$. But
  the exchange in situation $S$ is between $p$ and $u_i$ with $u_j$ in
  the middle which implies $u_j < u_i$ and therefore contradicts
  $u_j\geq u_i$.

  Let us now prove the other part of the lemma. With the notations of
  the lemma, we necessarily have $u_i>u_j$ so we have to prove that for
  any $s$ such that $i<s<j$, we have $u_s\leq u_j$. Note that every
  new pivot starts at the same position than the previous one or to
  its right. Moreover, as we proved above that there are no weighted
  values between both exchanged positions, the pivot exchanged with
  every weighted values to its right up to $u_j$ and $u_j$ was the
  first one. This implies that $u_j$ is greater than all the values
  between the pivot $u_i$ and itself.
\end{proof}

The previous lemma shows that the map $\phi_1$ is injective by proving
that we can find the previous pivot and the value with which it
is swapped at each step. We prove in next section that decreasing
weighted \SF s are enumerated by $n!$ which proves the following
proposition:

\begin{proposition}
  \label{propPhi1Bij}
  The map $\phi_1$ is a bijection.
\end{proposition}

Note that the inverse map comes straightforwardly from
Lemma~\ref{lemmaBij2And2}.

\subsubsection{The statistics through $\phi_1$}

Let us first define the new statistics on the decreasing weighted \SF
s.

\begin{definition}
  Let $(u,w)$ be a decreasing weighted {\SF} of size $n$.
  \begin{itemize}
  \item The \emph{number of inversions} $\inv(u,w)$ is the sum of the
    values of $u$ and $w$.
  \item The \emph{total weight} $\tw(u,w)$ is the sum of the values of
    $w$.
  \item The \emph{descent set} of $(u,w)$ is
    \begin{equation}
      \Des(u,w) = \{i~|~w_i > w_{i+1} \} \cup \{i~|~w_i = w_{i+1}, u_i >
      u_{i+1} \}
    \end{equation}
    and its \emph{descent composition} $\CDes(u,w)$ is the composition
    of $n$ whose descent set is $Des(u, w)$.
  \item The \emph{statistic $\LC$} on $(u,w)$ is the composition
    $\LC(u)$.
  \end{itemize}
\end{definition}

\begin{remark}
\label{remarkDecreasingLehmer}
 We shall make some remarks on the previous definitions.
 \begin{enumerate}
  \item The previous definitions are still correct if the weighted
    {\SF} is not decreasing.  Moreover, if we associate a null weight
    with a \SF, those definitions give the same statistics as the
    usual ones on \SF s.
  \item Note that on a decreasing \SF, the mirror composition of the
    statistic $\LC$ exactly corresponds to the composition whose
    descent set is the set of nonzero values of $u$. Hence, we have
    directly $\tw(u,w) = \inv(u,w) - \maj(\overline{\LC(u,w)})$.
 \end{enumerate}
\end{remark}

\begin{proposition}
\label{lemmaStatsPhi1}
 Let $u$ be a \SF, then:
 \begin{itemize}
  \item $\tw(\phi_1(u)) = \inv(u) - \maj(\overline{\LC(u)})$;
  \item $\CDes(\phi_1(u)) = \CDes(u)$;
  \item $\LC(\phi_1(u)) = \LC(u)$.
 \end{itemize}
\end{proposition}

To prove the last point of this proposition we shall consider the exchanges
of Algorithm~\ref{algoPhi1} as a succession of elementary exchanges
described by the following algorithm, where $i$ and $k$ come from the
notations of Algorithm~\ref{algoPhi1}.

\begin{algorithm}
\label{algoElementaryExchange}
 Set $j_1$ equal to $i$ and $j_2$ equal to $k$.
 \begin{itemize}
  \item Step 1: If $j_1 = k-1$, go to step 2. Otherwise, increment $u_{j_1+1}$
    and then swap $u_{j_1}$ and $u_{j_1+1}$. Set $j_1=j_1+1$ and redo
    Step 1.
  \item Step 2: If $j_2 = i$, the algorithm stops. Otherwise, decrement
    $u_{j_2-1}$ and then swap $u_{j_2}$ and $u_{j_2-1}$. Set
    $j_2 = j_2 - 1$ and redo Step 2.
 \end{itemize}
\end{algorithm}

For example, if at some point of Algorithm~\ref{algoPhi1}, we have
$u=\ldots4106\ldots$ and we have to exchange 6 with 4, the exchanges
of Algorithm~\ref{algoElementaryExchange} would be:
\begin{itemize}
 \item we begin with Step 1:
 \begin{itemize}
  \item $u=\ldots{\bf \textcolor{red}{24}}06\ldots$;
  \item $u=\ldots2{\bf \textcolor{red}{14}}6\ldots$;
 \end{itemize}
 \item and then apply Step 2:
 \begin{itemize}
  \item $u=\ldots21{\bf \textcolor{red}{63}}\ldots$;
  \item $u=\ldots2{\bf \textcolor{red}{60}}3\ldots$;
  \item $u=\ldots{\bf \textcolor{red}{61}}03\ldots$
 \end{itemize}
\end{itemize}

\begin{proof}[Proof of Proposition~\ref{lemmaStatsPhi1}]
 The proof of this proposition is based on the first point of
 Remark~\ref{remarkDecreasingLehmer} and works by proving that the
 statistics $\inv$, $\CDes$ and $\LC$ do not change at each exchange
 of Algorithm~\ref{algoPhi1}.
 
 The $\inv$ statistic is not modified since each decrementation of a
 value of $u$ is balanced by a incrementation of a value of $w$.

 In order to prove that the descent set of a weighted {\SF} does not
 change when we are doing an exchange of Algorithm~\ref{algoPhi1}, we
 have to study different cases depending on whether there is a descent
 at positions $i-1$, $i$, $k-1$, or $k$ or not. Let
 $(u^{(1)},w^{(1)})$ be the weighted {\SF} before the exchange and
 $(u^{(2)}, w^{(2)})$ the one after it.
 \begin{itemize}
  \item We start by considering what happens at position $k$.
  \begin{itemize}
   \item As $u_k^{(1)}$ is the rightmost occurrence of the $pivot$,
     we have $k \ne n$ as the pivot is nonzero and
     $u_k^{(1)} > u_{k+1}^{(1)}$ so the only way not to have a descent
     at $k$ is that $w_k^{(1)} < w_{k+1}^{(1)}$. But in this case
     Lemma~\ref{lemmaBij2And2} tells 
     us that the $pivot$ was at position $k+1$ at the previous step,
     so $u_i^{(1)} \leq u_{k+1}^{(1)} + 1$ and then after the exchange
     $u_k^{(2)} \leq u_{k+1}^{(2)}$ with $w_k^{(2)} \leq
     w_{k+1}^{(2)}$ so $k$ is not a descent.
   \item In any other case, $(u^{(1)},w^{(1)})$ has a descent in $k$,
     so $w^{(1)}_k \geq w^{(1)}_{k+1}$ and then $w^{(2)}_k >
     w^{(2)}_{k+1}$ so $(u^{(2)},w^{(2)})$ also has a descent in $k$.
  \end{itemize}
  \item If $i \neq k-1$, Lemma~\ref{lemmaBij2And2} implies that
    $w^{(1)}_{k-1} = 0$. Moreover, $u^{(1)}_{k-1} < u^{(1)}_k$ so $k-1$
    is not a descent for $(u^{(1)}, w^{(1)})$ and it is necessarily
    also the case for $(u^{(2)}, w^{(2)})$.
  \item If $i = k-1$, it is possible to have a descent or not in $i$.
  \begin{itemize}
   \item As $u^{(1)}_i < u^{(1)}_k$, the only way to have a descent in
     $i$ is to have $w^{(1)}_i > w^{(1)}_k$ and then $w^{(2)}_i \geq
     w^{(2)}_k$ with $u^{(2)}_i > u^{(2)}_k$.
   \item If $i$ is not a descent, $w^{(1)}_i \leq w^{(1)}_k$ and then
     $w_i^{(2)} < w_k^{(2)}$ so $i$ is not a descent either for
     $(u^{(2)}, w^{(2)})$.
  \end{itemize}
  \item For the descent in $i$ when $i < k-1$, we have $w^{(1)}_{i+1}
    = 0$ and $u^{(1)}_i > u^{(1)}_{i+1}$ so $i$ is always a descent
    of $(u^{(1)}, w^{(1)})$ in this case.  Moreover, we still have
    $w^{(2)}_{i+1} = 0$ with $u^{(2)}_i > u^{(2)}_{i+1}$ so it is also
    the case for $(u^{(2)}, w^{(2)})$.
  \item For the position $i-1$, suppose that $u_{i-1}^{(1)}$ is not strictly greater
    than the $pivot$. As $w^{(1)}_{i-1} = w^{(2)}_{i-1}$
    with $w^{(1)}_i = w^{(2)}_i$ and $u^{(2)}_{i-1} = u^{(1)}_{i-1}
    \leq u^{(1)}_i < u^{(2)}_i$, $(u^{(1)}, w^{(1)})$ has a descent in
    $i-1$ if and only if it is the case for $(u^{(2)}, w^{(2)})$. If
    $u_{i-1}^{(1)}$ is not strictly greater than the $pivot$, the same
    argument gives the same result.
 \end{itemize}
 We have shown that in any possible situation, the descent set is
 constant after each exchange of Algorithm~\ref{algoPhi1}.
 
 To prove that the $\LC$ statistic does not change at each step we
 prove that it is also true after each elementary exchange of
 Algorithm~\ref{algoElementaryExchange}.  As in the first step we have
 $u_{j_1} > u_{j_1 + 1}$ and in the second step we have $u_{j_2-1}
 \leq u_{j_2}$ we only have to prove that two \SF s $u^{(1)}$ and
 $u^{(2)}$ which differ only at positions $j$ and $j+1$ such that, if
 $u^{(1)} = \ldots ab\ldots$ with $a>b$, $u^{(2)} = \ldots
 b\hspace{-0.1cm}+\hspace{-0.13cm}1\hspace{0.12cm}a\ldots$, then
 $\LC(u^{(1)}) = \LC(u^{(2)})$. There are three
 different cases. Let $S$ be the set obtained during the computation
 of the two $\LC$ statistics before considering the value at position
 $j+1$.
 \begin{itemize}
  \item In $u^{(1)}$, if $b \leq |S|$ and $a \leq |S|$, then after
    considering those two values in $u^{(1)}$, the set used to compute
    $\LC(u^{(1)})$ is still $S$. In $u^{(2)}$ we also have $a \leq
    |S|$ and $b+1 \leq a \leq |S|$ so the set has not changed either.
  \item In $u^{(1)}$, if $b \leq |S|$ and $a > |S|$, let $\alpha$ be
    the $(a-|S|)$-th letter of $[1,n]\backslash S$, then after considering those two
    values, the set is equal to $S\cup \{\alpha\}$. In $u^{(2)}$, after
    considering $a$, the set is equal to $S\cup \{\alpha\}$ and
    then $b+1 \leq |S|+1$ so the set does not change.
  \item In $u^{(1)}$, if $b > |S|$ then the set is changed to $S\cup
    \{\beta\}$ where $\beta$ is the $(b-|S|)$-th letter of $[1,
      n]\backslash S$. Then as $a > b$, we have $a > |S| + 1$ and the
    set changes again to $S \cup \{\beta, \alpha\}$ where $\alpha$ is
    the $(a-|S|-1)$-th letter of $[1, n]\backslash
    (S\cup\{\beta\})$. For $u^{(2)}$, we still have $a > |S|$ so we
    add to $S$ the $(a-|S|)$-th value of $[1, n]\backslash S$ which is
    also $\alpha$ because we necessarily have $\beta < \alpha$. Then
    when we read $b+1$, we have $b + 1 > |S| +1$ so we add to the set
    the $(b - |S|)$-th value of $[1, n]\backslash (S\cup\{\alpha\})$
    which is again $\beta$, so the final set is the same in both
    cases.
 \end{itemize}
 This proves the last point of the proposition.
\end{proof}

\subsection{Decreasing weighted \SF s to weighted Dyck paths}
\label{subSectPhi2}

Let us now describe the final bijection between decreasing weighted
\SF s and weighted Dyck paths that corresponds to the third part of
Diagram~\eqref{eqSchemaBijs}.

\begin{equation*}
  \textcolor{light-gray}{\text{P} \overset{\psi_{FV}}{\longleftrightarrow}
  \text{WDP} \overset{\psi}{\longleftrightarrow}} \text{WDP}
  \overset{\phi_{2}}{\longleftrightarrow} \text{DWSF}
  \textcolor{light-gray}{\overset{\phi_{1}}{\longleftrightarrow} \text{SF}}
\end{equation*}
To describe it, we build a Dyck path
from a decreasing {\SF} and carry the weight without modifying
it. Then we show that it
sends the triple of statistics of the decreasing weighted \SF s to the
triple of statistics of the weighted Dyck paths.

\subsubsection{Description of the bijection $\phi_2$ between decreasing weighted \SF s and weighted Dyck paths}

The map $\phi_2$ on a decreasing weighted {\SF} $(u,w)$ of size~$n$ is
defined by carrying the weight $w$ and building the Dyck path $D$ from
the decreasing {\SF} $u$ using the following algorithm. 
\begin{algorithm}
\label{algoPhi2}
 Set $D_1 := /$ and $D_{2n} := \backslash$.  Then, for each $i$ in
 $\{1, \ldots, n-1\}$,
 \begin{itemize}
  \item set $D_{2i} := \backslash$ if $n-i$ is a value in $u$ and
    $D_{2i} := /$ otherwise;
  \item set $D_{2i+1} := \backslash$ if $u_i = 0$ and $D_{2i+1} := /$
    otherwise.
 \end{itemize}
\end{algorithm}

An example is given in Figure~\ref{figPhi2}. The positions of the
zeroes in the {\SF} are $\{3,5,6,8,9\}$ which correspond to the
positions $i$ such that $D_{2i+1} = \backslash$. The values correspond
to $n-i$ where $i$ are the positions where $D_{2i} = \backslash$.

\begin{figure}[h!t]
  \begin{center}
    \begin{tikzpicture}
      \node (eq) at (0,0) {$\phi_2(540300200, 002201000) = $};
      \node (F1) at (6.3,0) {
        \scalebox{0.4}{
          \begin{tikzpicture}
  \draw[dotted, thick, color=gray!60] (0, 0) grid (18, 6);
  \draw[rounded corners=1, color=black, line width=1] (0, 0) -- (1, 1)
  -- (2, 2) -- (3, 3) -- (4, 4) -- (5, 5) -- (6, 6) -- (7, 5) -- (8,
  4) -- (9, 5) -- (10, 4) -- (11, 3) -- (12, 2) -- (13, 1) -- (14, 0)
  -- (15, 1) -- (16, 2) -- (17, 1) -- (18, 0);

  \draw[dotted, very thick] (2, 0) -- (2, 6); \draw[dotted, very
    thick] (4, 0) -- (4, 6); \draw[dotted, very thick] (6, 0) -- (6,
  6); \draw[dotted, very thick] (8, 0) -- (8, 6); \draw[dotted, very
    thick] (10, 0) -- (10, 6); \draw[dotted, very thick] (12, 0) --
  (12, 6); \draw[dotted, very thick] (14, 0) -- (14, 6); \draw[dotted,
    very thick] (16, 0) -- (16, 6);
  
  \node (pond1) at (5, 5.5) {2}; \node (pond2) at (7, 5.5) {2}; \node
  (pond2) at (11, 3.5) {1};
\end{tikzpicture}
 }};
    \end{tikzpicture}
    \caption{An example of $\phi_2$.}
    \label{figPhi2}
  \end{center}
\end{figure}
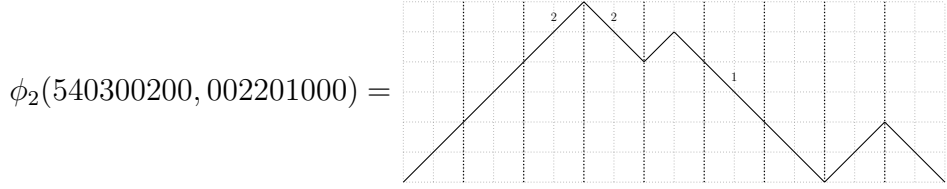

\begin{lemma}
\label{lemmaDefPhi2}
 In the path $D$ built in Algorithm~\ref{algoPhi2}, let $h_i$ be
 the height be the height of the $2i$-th step of $D$. Then, for all $i$, we
 have $(h_i-1)/2$ equal to the number of nonzero values smaller than or equal
 to $n-i$ in $u$ to the left of $u_i$.
\end{lemma}

\begin{proof}
  Let us call $W$ the maximal weight associated with $u$. Our aim is to
  prove that for all $i$ we have $(h_i-1)/2=W_i$.
  
  We prove this lemma by induction. For $i=1$, we have
  $(h_1-1)/2=0=W_0$  and the property holds. Assume the property for a
  given $i$. We have $W_{i+1} = W_{i} + 1$ if and only if $n-i$ is not
  a value in $u$ and $u_i \ne 0$ so that all the values on the left of
  $u_i$ that count for $W_i$ also count for $u_{i+1}$ and it is also
  the case for $u_i$. In this case we have $D_{2i}=D_{2i+1}=/$ and so
  $h_{i+1} = h_i + 2$ so the property holds. With the same idea, we
  have $W_{i+1} = W_{i} - 1$ if and only if $n-i$ is a value in $u$
  and $u_i = 0$, in this case we have $h_{i+1} = h_i - 2$. Finally, in
  the other two cases we have $W_i = W_{i+1}$ and
  $h_i=h_{i+1}$.
\end{proof}

\begin{proposition}
 The map $\phi_2$ is well-defined from decreasing weighted \SF s to
 weighted Dyck paths and is a bijection.
\end{proposition}

\begin{proof}
 The image of a decreasing weighted {\SF} is a
 path. Lemma~\ref{lemmaDefPhi2} proves that the height of the path is
 always nonnegative and that $h_n = 1$.  As $D_{2n} = \backslash$, the
 path is a Dyck path. Lemma~\ref{lemmaDefPhi2} also gives us that the
 constraints on the weight of a decreasing weighted {\SF} correspond
 to the constraints of its image.
 
 As the constraints for the weights correspond exactly from one object
 to the other, we only need to prove that $\phi_2$ is a bijection from
 decreasing \SF s to Dyck paths to prove that $\phi_2$ is a bijection
 on the weighted objects.
 
 As decreasing \SF s and Dyck paths are both enumerated by the Catalan
 numbers, we only need to prove that this map is injective. The only
 way to have the same image from two decreasing \SF s is that the
 nonzero values and their positions are fixed, but as the \SF s are
 decreasing, there is no choice in the order of those values.
\end{proof}

Note that by proving that $\phi_2$ is a bijection, we proved that
decreasing weighted \SF s are enumerated by $n!$, so this finishes the
proof of Proposition~\ref{propPhi1Bij}.

\subsubsection{The statistics through $\phi_2$}

\begin{proposition}
\label{lemmaStatsPhi2}
 Let $(u,w)$ be a decreasing weighted {\SF}. We have
 \begin{itemize}
  \item $\tw(u,w) = \tw(\phi_2(u,w))$;
  \item $\LC(u,w) = \GC^0(\phi_2(u,w))$;
  \item $\CDes(u,w) = \CDes^0(\phi_2(u,w))$.
 \end{itemize}
\end{proposition}

\begin{proof}
 As the weight is carried without being changed, we have $\tw(u,w) = \tw(\phi_2(u,w))$. For
 the descent set, as the weight does not change and the positions of the
 zeroes in $u$ correspond to the $\backslash$ steps at odd positions
 in $D$, we also have $\CDes(u,w) = \CDes(\phi_2(u,w))$. Moreover,
 $\LC(u,w)$ is the mirror composition of the composition whose descent
 set is the nonzero values in $u$ and $\GC^0(D,w)$ is related to the
 $\backslash$ steps at even positions in $D$, so that we also have
 $\LC(u,w) = \GC^0(\phi_2(u,w))$.
\end{proof}

\section{Main result}
\label{sectionGlobal}

\subsection{Proof of previous conjectures}

We now have all the tools we need to prove the conjectures
of~\cite{NTW}. To do so, we prove Conjecture~6.2 which is the
equidistribution of both triples of statistics.

Let $\Phi = \psi_{FV}^{-1}\circ\psi^{-1}\circ\phi_2\circ\phi_1$.

\begin{theorem}
  \label{globalTh}
 Let $I$ and $J$ be two compositions of $n$. We have
 \begin{equation}
  \sum_{\substack{u \in \SFE_n \\ \CDes(u) = I \\ \LC(u) = J}}
  q^{\inv(u)- \maj(\overline{J})} = \sum_{\substack{\sigma \in \SG_n
      \\ \Rec(\sigma) = I \\ \GC(\sigma) = J}} q^{\tto(\sigma)}.
 \end{equation}
\end{theorem}

\begin{proof}
  As a composition of bijections, the map $\Phi$ is a bijection
  between {\SF} and permutations. Let $u$ be a {\SF}. Applying
  Propositions~\ref{lemmaStatsPhi1}, \ref{lemmaStatsPhi2},
  \ref{lemmaStatsInvol}, and \ref{lemmaStatsFV} give
  \begin{equation}
    \begin{array}{rcl}
      \LC(u) & = & \GC(\Phi(u)) \\
      \CDes(u) & = & \Rec(\Phi(u)) \\
     \inv(u) - \maj\left(\overline{\LC(u)}\right) & = & \tto(\Phi(u))
    \end{array}
  \end{equation}
  which proves the theorem.
\end{proof}

\begin{remark}
  The other conjectures of~\cite{NTW} are directly obtained from this
  one.
\end{remark}

Proving those conjectures proves that it is possible to interpret the
entries of
the transitions matrices of~\cite{NTW} in terms of the triple of
statistics $(\GC,\Rec,\tto)$. It also gives another $q$-refinement of the
steady-state probabilities of the PASEP with statistics that arise from
the combinatorics of the PASEP.

\subsection{Other properties of this bijection}

In addition to the three statistics we have studied, the bijection we
defined from subexcedent functions to permutations also carries another
statistic.

\begin{definition}
  Define the \emph{left-to-right maxima} of a permutation as the
  values with only smaller values to their left.

  The same statistic on subexcedent functions is defined as the
  positions containing zeroes.
\end{definition}

For example, with $u = 315503200$ as with $\sigma=\Phi(u)=528713649$ the
right to left maxima are $5$, $8$, and $9$.

Note that the definition of the left-to-right maxima of a subexcedent function
corresponds to its definition on permutations after taking the Lehmer
code of the inverse of the permutation.

\begin{proposition}
  \label{LRM}
  Let $u$ be a subexcedent function. Its left-to-right maxima
  corresponds to the left-to-right maxima of $\Phi(u)$.
\end{proposition}

To prove this proposition, we need to see what happens to this
statistic through the different bijections. First we need the
following lemma that gives an equivalent definition of the left-to-right
maxima in terms of descents and $\312$ patterns.

\begin{lemma}
  \label{equivLRM}
  Let $\sigma$ be a permutation. A value $k = \sigma_j$ is 
  a left-to-right maximum of $\sigma$ if and only if $\sigma_{j-1} < \sigma_j$ and
  $\tto_k(\sigma)=0$.
\end{lemma}

\begin{proof}[Proof of Lemma \ref{equivLRM}]
  If $k$ is a left-to-right maximum of $\sigma$, we necessarily have
  $\sigma_{j-1} < \sigma_j$ and there are no values greater than $k$
  to its left. So $\tto_k(\sigma)=0$.

  Conversely, assume there exists $i < j$ such that
  $\sigma_i>\sigma_j$ and let $i$ be the rightmost position such that
  $\sigma_i>\sigma_j$  with $i < j$. Then
  $\sigma_i>\sigma_{i+1}<\sigma_j$
  and this is a $\312$ pattern where $k$ stands for the $2$.
\end{proof}

With this lemma we are now able to follow what happens to the
left-to-right maxima through the various bijections.

\begin{proof}[Proof of Proposition~\ref{LRM}]
  Let $\sigma$ be a permutation and $k$ a value of $\sigma$.
  Lemma~\ref{equivLRM} shows that if we apply the Fran\c con-Viennot
  bijection, the value $k$ is a
  right to left maximum if and only if $D_{2k}=\backslash$ and $w_k=0$
  where $(D,w)=\psi_{FV}(\sigma)$. Thanks to Lemma~\ref{otherDefPsi}
  the value $k$ is a right to left maximum if and only if
  $D^0_{2k+1}=\backslash$ and $w^0_k= 0$ where
  $(D^0,w^0)=\psi\circ\psi_{FV}(\sigma)$. Moreover, let $(u, w)$ be
  the nondecreasing weighted subexcedent function obtained after
  applying $\phi_2^{-1}$ to $(D^0, w^0)$. Then $k$ is a right to left
  maximum of $\sigma$ if and only if $u_k=0$ and $w_k=0$.

  Finally, we need to prove that a value at
  position $k$ in a subexcedent function $v$ is equal to zero if and only
  if the value at the same position after applying the sorting
  function $\phi_2$ is also zero with a null weight. In
  Algorithm~\ref{algoPhi1} one can see that no pivot shall exchange
  with a value equal to zero so the result at position $k$ shall also
  be zero with a null weight. Conversely, if the value at position $k$
  in $\phi_1(u)$ has a weight equal to zero, it means that the value
  never changed during the algorithm so the value is also equal to zero in
  $u$.
\end{proof}

\section{A variation on the statistics on permutations}
\label{sectionOtherConv}
\subsection{Another combinatorial interpretation of ${\goth M}^{(n)}(q)$}

In this part of the paper we change the convention on permutations
such that $\sigma_0 = \sigma_{n+1} = 0$. In the definition of the
Genocchi descent set, we compare each value to its successor so we
need to define a new statistic associated with this new convention.

\begin{definition}
 Let $\sigma \in \SG_n$. We define the \emph{Genocchi descent set of
   type 0} as
 \begin{equation}
  \GDes^0(\sigma) = \{\sigma_i~|~i \in \{1, \ldots,n\},~ \sigma_i>\sigma_{i+1}
  \}.
 \end{equation}
 Note that $\sigma_n$ and $n$ always belong to $\GDes^0(\sigma)$. We
 also define the \emph{Genocchi composition of descents of type 0}
 (denoted by $\GC^0$) as the composition of $n$ whose descent set is
 $\GDes^0(\sigma)\backslash\{n\}$.
\end{definition}

For example with $\sigma = 528971364$, we have $\GDes^0(\sigma) = \{
4,5,6,7,9\}$ and $\GC^0(\sigma) = (4,1,1,1,2)$.

The Fran\c con-Viennot bijection still exists with this convention and
builds a large Laguerre history. We call $\psi^0_{FV}$ the application
constructing a weighted Dyck path from a permutation with this
convention.
\begin{definition}
  Let $\sigma \in \SG_n$, define $\psi^0_{FV}(\sigma)$ as the path
  obtained by using Algorithm~\ref{algoFV} for the values $1$ to $n-1$
  in $\sigma$ and then adding an increasing step at the beginning of
  the path and a decreasing step at the end. The Algorithm builds a
  weight of size $n-1$, we add a $0$ at the end to obtain a weight of
  size $n$.
\end{definition}
Note that we do not need to consider $n$ in the construction as it is
always greater than both its neighbors with this convention.
An example is given in Figure~\ref{figFV0} with $\sigma =528971364$
alongside with $\psi_{FV}(\sigma)$.

\begin{figure}[h!t]
  \begin{center}
    \begin{tikzpicture}[scale=0.5]
      \node (F1) at (0,0) {
        \scalebox{0.3}{
          \begin{tikzpicture}
  \draw[dotted, thick, color=gray!60] (0, 0) grid (18, 6);
  \draw[rounded corners=1, color=black, line width=1] (0, 0) -- (1, 1)
  -- (2, 2) -- (3, 3) -- (4, 4) -- (5, 5) -- (6, 6) -- (7, 5) -- (8,
  4) -- (9, 5) -- (10, 4) -- (11, 3) -- (12, 2) -- (13, 1) -- (14, 0)
  -- (15, 1) -- (16, 2) -- (17, 1) -- (18, 0);

  \draw[dotted, very thick] (2, 0) -- (2, 6); \draw[dotted, very
    thick] (4, 0) -- (4, 6); \draw[dotted, very thick] (6, 0) -- (6,
  6); \draw[dotted, very thick] (8, 0) -- (8, 6); \draw[dotted, very
    thick] (10, 0) -- (10, 6); \draw[dotted, very thick] (12, 0) --
  (12, 6); \draw[dotted, very thick] (14, 0) -- (14, 6); \draw[dotted,
    very thick] (16, 0) -- (16, 6);
  
  \node (pond1) at (5, 5.5) {2}; \node (pond2) at (7, 5.5) {2}; \node
  (pond2) at (11, 3.5) {1};
\end{tikzpicture}
 }};
      \node (F2) at (12,0) {
        \scalebox{0.3}{
          \begin{tikzpicture}
  \draw[dotted, thick, color=gray!60] (0, 0) grid (18, 6);
  \draw[rounded corners=1, color=black, line width=1] (0, 0) -- (1, 1)
  -- (2, 2) -- (3, 3) -- (4, 4) -- (5, 5) -- (6, 4) -- (7, 5) -- (8,
  6) -- (9, 5) -- (10, 4) -- (11, 3) -- (12, 2) -- (13, 1) -- (14, 2)
  -- (15, 1) -- (16, 2) -- (17, 1) -- (18, 0);

  \draw[dotted, very thick] (2, 0) -- (2, 6); \draw[dotted, very
    thick] (4, 0) -- (4, 6); \draw[dotted, very thick] (6, 0) -- (6,
  6); \draw[dotted, very thick] (8, 0) -- (8, 6); \draw[dotted, very
    thick] (10, 0) -- (10, 6); \draw[dotted, very thick] (12, 0) --
  (12, 6); \draw[dotted, very thick] (14, 0) -- (14, 6); \draw[dotted,
    very thick] (16, 0) -- (16, 6);
  
  \node (pond1) at (5, 5.5) {2}; \node (pond2) at (7, 5.5) {2}; \node
  (pond2) at (11, 3.5) {1};
\end{tikzpicture}
 }};
    \end{tikzpicture}
    \caption{An example of $\psi^0_{FV}$ on the left and $\psi_{FV}$
      on the right with $\sigma = 528971364$.}
    \label{figFV0}
  \end{center}
\end{figure}
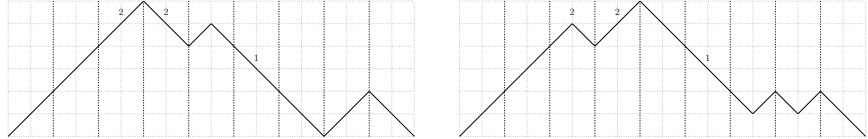

The inverse map is built as follows. Apply $\psi_{FV}^{-1}$ to
the path obtained after removing the first and last steps. At the end,
instead of removing the last $\circ$ in the obtained word, replace it
by $n$.

\begin{proposition}
 Let $\sigma \in \SG_n$, we have:
 \begin{itemize}
  \item $\GC^0(\sigma) = \GC^0(\psi_{FV}^0(\sigma))$;
  \item $\Rec(\sigma) = \CDes^0(\psi_{FV}^0(\sigma))$;
  \item $\tto(\sigma) = \tw(\psi_{FV}^0(\sigma))$.
 \end{itemize}
\end{proposition}

The proof of this statement is the same as the
one of Proposition~\ref{lemmaStatsFV}, the only difference being that
we have to add one to the index of the steps in the path as we add an
extra step at the beginning when applying the Fran\c con-Viennot
bijection with this convention.

From this proposition we derive the following proposition.

\begin{corollary}
 Let $I$ and $J$ be two compositions of $n$, we have:
 \begin{equation}
  \sum_{\Rec(\sigma) = I \atop \GC^0(\sigma) = J} q^{\tto(\sigma)} =
  \sum_{\Rec(\sigma) = I \atop \GC(\sigma) = J} q^{\tto(\sigma)}.
 \end{equation}
\end{corollary}

The proof of this statement is based on the same approach as the
proof of Theorem~\ref{globalTh}, using the bijection
$\Psi=(\psi_{FV}^0)^{-1} \circ \psi \circ \psi_{FV}$.

This result gives another combinatorial interpretation of the
entries of the transition matrices described previously. It also
implies another way of refining the PASEP in terms of statistics close
to the usual one used to describe its combinatorial behavior.

\subsection{Binary search trees}

By construction, $\Psi$ preserves the recoil classes (that is the
permutations having the same recoils composition). It happens that
$\Psi$ preserves smaller sets known as co-sylvester classes
of permutations, which have been described in~\cite{PBT} with
sylvester classes. A co-sylvester class is represented by a binary
search tree.

\begin{definition}
  A \emph{binary search tree} is a labeled binary tree such that for each
  node, all the values in the left (resp. right) subtree are smaller
  (resp. greater) than the value at the node.
\end{definition}

\begin{definition}
  The \emph{restriction} of a binary search tree to an interval is
  the tree obtained after removing the values that are not in the
  interval then erasing the subtrees with no nodes and merging the edges
  with no values between them.
\end{definition}

An example of restriction of a binary tree to an interval is given in
Figure~\ref{restriction}.

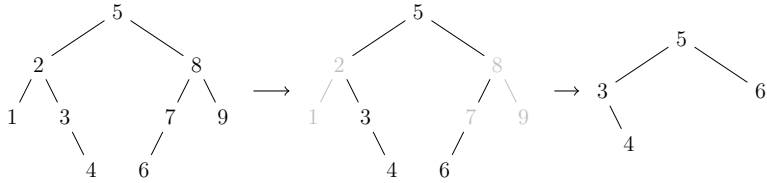
\begin{figure}[h!t]
  \begin{center}
    \begin{tikzpicture}
      \node (F1) at (6.3,0) {
          \begin{tikzpicture}[scale=1]
  \node (T1) at (0,0) {\scalebox{0.7}{\begin{tikzpicture}[scale=1]
  \node(N5) at (0,0){5};
  \node(N2) at (-1.5,-1){2};
  \node(N8) at (1.5,-1){8};
  \node(N1) at (-2,-2){1};
  \node(N3) at (-1,-2){3};
  \node(N7) at (1,-2){7};
  \node(N9) at (2,-2){9};
  \node(N4) at (-0.5,-3){4};
  \node(N6) at (0.5,-3){6};
  
  \draw (N5) -- (N2);
  \draw (N5) -- (N8);
  \draw (N1) -- (N2);
  \draw (N3) -- (N2);
  \draw (N3) -- (N4);
  \draw (N7) -- (N8);
  \draw (N8) -- (N9);
  \draw (N6) -- (N7);
\end{tikzpicture}
 }};
  \node (T2) at (4,0) {\scalebox{0.7}{\begin{tikzpicture}[scale=1]
  \node(N5) at (0,0){5};
  \node(N2) at (-1.5,-1){\color{gray!50}2};
  \node(N8) at (1.5,-1){\color{gray!50}8};
  \node(N1) at (-2,-2){\color{gray!50}1};
  \node(N3) at (-1,-2){3};
  \node(N7) at (1,-2){\color{gray!50}7};
  \node(N9) at (2,-2){\color{gray!50}9};
  \node(N4) at (-0.5,-3){4};
  \node(N6) at (0.5,-3){6};
  
  \draw (N5) -- (N2);
  \draw (N5) -- (N8);
  \draw[color=gray!50] (N1) -- (N2);
  \draw (N3) -- (N2);
  \draw (N3) -- (N4);
  \draw (N7) -- (N8);
  \draw[color=gray!50] (N8) -- (N9);
  \draw (N6) -- (N7);
\end{tikzpicture}
 }};
  \node (T3) at (7.5,0) {\scalebox{0.7}{\begin{tikzpicture}[scale=1]
  \node(N5) at (0,0){5};
  \node(N2) at (-1.5,-1){3};
  \node(N8) at (1.5,-1){6};
  \node(N3) at (-1,-2){4};
  
  \draw (N5) -- (N2);
  \draw (N5) -- (N8);
  \draw (N3) -- (N2);
\end{tikzpicture}
 }};

  \draw[->] (T1) -- (T2);
  \draw[->] (T2) -- (T3);
\end{tikzpicture}
 };
    \end{tikzpicture}
    \caption{Restriction of a binary search tree to the interval
      $\{3, 4, 5, 6\}$.}
    \label{restriction}
  \end{center}
\end{figure}

By analogy, the restriction of a permutation to an interval
is the word obtained by removing the values that are not in the
interval. For example, $528971364_{|_{\{3,4,5,6\}}}=5364$.

We associate a binary search tree with a permutation through
the following algorithm.

\begin{algorithm}
  \label{binaryPerm}
  At the beginning set $I= \{1,2,\cdots,n\}$. The root of the current
  tree is the value $k$ of $I$ such that $k$ is the first value of
  $\sigma_{|_I}$. Apply this
  algorithm to the interval of values smaller than $k$ in $I$ to
  obtain the left subtree and to the values greater than $k$ in $I$
  for the right subtree.
\end{algorithm}
Denote by $\BST(\sigma)$ the map associated with this algorithm.

An example of this algorithm is given in
Figure~\ref{figBinaryTree}. In $\sigma = 528971364$ the first value is
$5$. The left subtree is built with $\sigma_{|_{\{1,2,3,4\}}} = 2134$
and the right subtree is built with $\sigma_{|_{\{6,7,8,9\}}} = 8976$.

\begin{remark}
  The classical way to build the tree is to read the permutation from
  left to right and add each value to its only possible position as a
  leaf in the binary search tree. So with $\sigma = 528971364$, start
  with a $5$, then $2$ must be to the left of $5$, $8$ to the right of
  $5$, $9$ to the right of $5$ and of $8$, and so on. In the rest of
  the paper we shall only need the description of
  Algorithm~\ref{binaryPerm}.
\end{remark}

\begin{figure}[h!t]
  \begin{center}
    \begin{tikzpicture}
      \node (F1) at (0,0) {
          \begin{tikzpicture}[scale=1]
  \node(N5) at (0,0){5};
  \node(N2) at (-1.5,-1){2};
  \node(N8) at (1.5,-1){8};
  \node(N1) at (-2,-2){1};
  \node(N3) at (-1,-2){3};
  \node(N7) at (1,-2){7};
  \node(N9) at (2,-2){9};
  \node(N4) at (-0.5,-3){4};
  \node(N6) at (0.5,-3){6};
  
  \draw (N5) -- (N2);
  \draw (N5) -- (N8);
  \draw (N1) -- (N2);
  \draw (N3) -- (N2);
  \draw (N3) -- (N4);
  \draw (N7) -- (N8);
  \draw (N8) -- (N9);
  \draw (N6) -- (N7);
\end{tikzpicture}
 };
    \end{tikzpicture}
    \caption{The binary search tree associated with $\sigma = 528971364$.}
    \label{figBinaryTree}
  \end{center}
\end{figure}
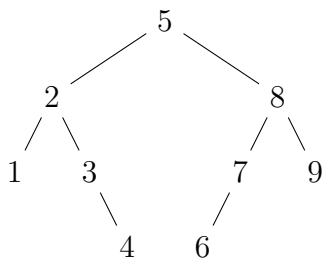

We say that two permutations $\sigma$ and $\tau$ belong to the same
co-sylvester class if their associated binary search trees are
equal.

To prove the equality of binary search trees, it shall be 
convenient to use the description of the following lemma.

\begin{lemma}
  \label{lemBinarySearch}
  Let $A$ and $A'$ be two binary search trees. We have $A = A'$ if and
  only if they have the same root in their restriction to any
  interval.
\end{lemma}

\begin{proof}[Proof of Lemma \ref{lemBinarySearch}]
  If $A = A'$, the property is clearly true. Let $A$ and $A'$ be two
  binary search trees such that for any interval $I$ they give the same
  root when restricted to $I$. In particular $A$ and $A'$ have the same root (if
  we take $I$ to be the interval of all values in them). The proof of
  the lemma comes by induction by applying the property on the
  intervals of values smaller than the root and greater than the root
  which gives the two subtrees of the root.
\end{proof}

\subsection{A co-sylvester class preserving bijection}

Our aim is to prove the following theorem with
$\Psi=(\psi_{FV}^0)^{-1} \circ \psi \circ \psi_{FV}$.
\begin{theorem}
  \label{thmSylvester}
  Let $\sigma$ be a permutation. Then $\sigma$ and $\Psi(\sigma)$ belong
  to the same co-sylvester class, \ie, have the same binary search tree.
\end{theorem}

In order to prove this theorem, we need to build binary search trees
associated with weighted Dyck paths corresponding to the one obtained
from a permutation before applying $\psi_{FV}$ and $\psi_{FV}^0$. Let
us start with the algorithm associated with $\psi_{FV}$. Its main idea
came after discussing with Aval, Boussicault, and Zubieta~\cite{ABZ}.

\begin{algorithm}
  \label{DyckSylvester}
  Start with $I = \{1, 2, \cdots, n\}$.
  \begin{itemize}
  \item Among the $i\in I$ such that $w_i$ is minimal, the root $r$ of
    the current tree is equal to the smallest $i$ such that
    $D_{2i} = \backslash$.
  \item If there are no such positions, let $r$ be the maximal value
    of $I$ among the ones of minimal weight.
  \end{itemize}
  Apply
  this algorithm to the interval of values smaller than $r$ in $I$ to
  obtain the left subtree and to the values greater than $r$ for the
  right subtree.
\end{algorithm}
Denote by $\BST(D, w)$ the map associated with this algorithm.

An example of a Dyck path and its corresponding tree is given in
Figure~\ref{treeDyck}. At the first step, the $i$ of minimum weight
are $\{1, 2, 5, 7, 8, 9\}$. Among these, $5$ is the minimal
value such that $D_{2i}= \backslash$. To obtain the left subtree, we
apply the algorithm to $I=\{1, 2, 3, 4\}$ on the path. The minimal
weights are obtained for $i$ in $\{1, 2\}$ but there are no $i$ in this set such
that $D_{2i} = \backslash$ so~$r=2$. Apply
this algorithm respectively on $\{1\}$, $\{3, 4\}$, and $I=\{6,7,8,9\}$
to get the other parts of the tree.

\begin{figure}[h!t]
  \begin{center}
    \begin{tikzpicture}
      \node (F1) at (0,0) {
        \scalebox{0.4}{
          \begin{tikzpicture}
  \draw[dotted, thick, color=gray!60] (0, 0) grid (18, 6);
  \draw[rounded corners=1, color=black, line width=1] (0, 0) -- (1, 1)
  -- (2, 2) -- (3, 3) -- (4, 4) -- (5, 5) -- (6, 4) -- (7, 5) -- (8,
  6) -- (9, 5) -- (10, 4) -- (11, 3) -- (12, 2) -- (13, 1) -- (14, 2)
  -- (15, 1) -- (16, 0) -- (17, 1) -- (18, 0);

  \draw[dotted, very thick] (2, 0) -- (2, 6); \draw[dotted, very
    thick] (4, 0) -- (4, 6); \draw[dotted, very thick] (6, 0) -- (6,
  6); \draw[dotted, very thick] (8, 0) -- (8, 6); \draw[dotted, very
    thick] (10, 0) -- (10, 6); \draw[dotted, very thick] (12, 0) --
  (12, 6); \draw[dotted, very thick] (14, 0) -- (14, 6); \draw[dotted,
    very thick] (16, 0) -- (16, 6);
  
  \node (pond1) at (5, 5.5) {2}; \node (pond2) at (7, 5.5) {2}; \node
  (pond2) at (11, 3.5) {1};
\end{tikzpicture}
 }};
      \node (F1) at (6.3,0) {
        \scalebox{0.8}{
          \begin{tikzpicture}[scale=1]
  \node(N5) at (0,0){5};
  \node(N2) at (-1.5,-1){2};
  \node(N8) at (1.5,-1){8};
  \node(N1) at (-2,-2){1};
  \node(N3) at (-1,-2){3};
  \node(N7) at (1,-2){7};
  \node(N9) at (2,-2){9};
  \node(N4) at (-0.5,-3){4};
  \node(N6) at (0.5,-3){6};
  
  \draw (N5) -- (N2);
  \draw (N5) -- (N8);
  \draw (N1) -- (N2);
  \draw (N3) -- (N2);
  \draw (N3) -- (N4);
  \draw (N7) -- (N8);
  \draw (N8) -- (N9);
  \draw (N6) -- (N7);
\end{tikzpicture}
 }};
    \end{tikzpicture}
    \caption{A weighted Dyck path and its associated binary search
      tree.} 
    \label{treeDyck}
  \end{center}
\end{figure}
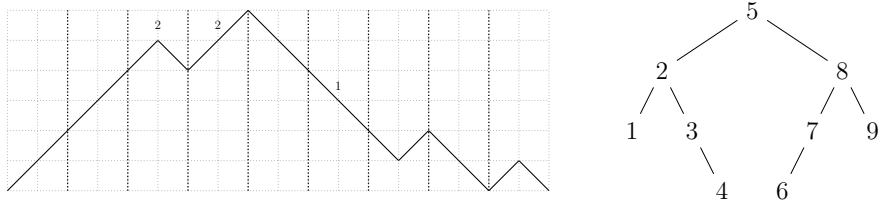

\begin{remark}
  Note that we could have applied this algorithm on any weighted path
  with $2n$ steps. It does not need to stay above the horizontal axis
  nor to end at position $(2n,0)$.
\end{remark}

\begin{lemma}
  \label{equalTrees}
  Let $\sigma$ be a permutation. Then
  $\BST(\sigma) = \BST(\psi_{FV}(\sigma))$.
\end{lemma}

In order to prove this lemma we need the following lemma.

\begin{lemma}
  \label{lemma312}
  Let $\sigma$ be a permutation and $i<j<k$ be such that
  $\sigma_j<\sigma_k<\sigma_i$ (\ie, they form a $3\!-\!1\!-\!2$ pattern), then
  $\sigma_k$ stands for the $2$ in a $\312$ pattern where the $1$ is to
  the right of $\sigma_i$.
\end{lemma}

\begin{proof}[Proof of Lemma~\ref{lemma312}]
  Among the values greater than $\sigma_k$ on the left of~$\sigma_j$,
  let~$\sigma_r$ be the rightmost one. Necessarily,
  $\sigma_r>\sigma_k>\sigma_{r+1}$ and $r \geq i$ so it defines a
  $\312$ pattern of $\sigma$ to the right of $\sigma_i$ where
  $\sigma_k$ is the $2$.
\end{proof}

\begin{proof}[Proof of Lemma~\ref{equalTrees}]
  To prove, we use
  Lemma~\ref{lemBinarySearch}. Let $A=\BST(\sigma)$ and
  $A'=\BST(D,w))$, where $(D, w) =
  \psi_{FV}(\sigma)$. Let $I$ be an interval of values of $\sigma$.

  The root of $A_{|_I}$
  corresponds to the first value $\sigma_{|_I}$ since one easily sees
  that $A_{|_I}=\BST(\sigma_{|_I})$ (see an example on Figure~\ref{restriction}). We
  have to prove that this value corresponds to the position returned
  by Algorithm~\ref{DyckSylvester} for the same interval on $(D,w)$.
  In the rest of the proof, consider $\sigma = \cdots \sigma_i \cdots
  \sigma_j \cdots$ where $\sigma_i$ is the first value of $I$ in
  $\sigma$ and $\sigma_j$ is another value of $I$.

  In order to prove
  that we have to search a step of minimal weight, we need to prove
  that for any $\312$ pattern where $\sigma_i$ is the $2$, we have a
  $\312$ pattern where $\sigma_j$ is also the $2$. Indeed, for any
  $\312$ pattern where $\sigma_i$ is the $2$, the value representing
  the $3$ (resp. the $1$) is necessarily greater (resp. smaller) than
  all values in $I$ so $\sigma_j$ is also between them and so is
  involved as a $2$ in a $\312$ pattern.

  If $\sigma_i > \sigma_{i-1}$ (so $D_{2i}=\backslash$) and if
  $\sigma_{j-1}<\sigma_{j}$ with $\sigma_j < \sigma_{i}$ then applying
  Lemma~\ref{lemma312} proves that there is a $\312$ pattern to the
  right of $\sigma_i$ where
  $\sigma_j$ stands for the $2$ and so the weight associated with
  $\sigma_j$ in $(D, w)$ is strictly greater than the one associated
  with $\sigma_i$. This proves that if $D_{2i} = \backslash$, it is the
  leftmost decreasing step in even position of minimal weight in the
  interval $I$.

  If $\sigma_i < \sigma_{i-1}$ (so $D_{2i} = /$) and if
  $\sigma_j>\sigma_i$ then $\sigma_{i-1}\sigma_i\cdots\sigma_j$ is a $\312$
  pattern because $\sigma_{i-1}$ is greater than any value in
  $I$. This proves that
  $D_{2i}$ is the rightmost even step among the ones of minimal
  weight. Moreover, if $\sigma_i > \sigma_j$ and
  $\sigma_{j-1}<\sigma_j$, we can again use
  Lemma~\ref{lemma312} to prove that there is a $\312$ pattern where
  $\sigma_j$ is a $2$ and where the $1$ is on the right of $\sigma_i$,
  so $w_i < w_j$ and there are no decreasing steps of minimal
  weight in position $2k$ for $k$ in $I$.

  We use Lemma~\ref{equalTrees} to prove that $A=A'$.
\end{proof}

Similarly we define the map $\BST^0$ building a binary search tree
from a weighted Dyck path equal to the binary search tree associated
with a permutation before applying $\psi_{FV}^0$.

\begin{definition}
  Let $(D,w)$ be a weighted Dyck path. Let $D'$ be the path obtained
  by removing the first step of $D$ and adding a decreasing step at
  its end. Define $\BST^0(D,w)$ as the result of $\BST(D', w)$.
\end{definition}

\begin{figure}[h!t]
  \begin{center}
    \begin{tikzpicture}
      \node (F1) at (0,0) {
        \scalebox{0.35}{
          \begin{tikzpicture}
  \draw[dotted, thick, color=gray!60] (0, 0) grid (18, 6);
  \draw[rounded corners=1, color=black, line width=1] (0, 0) -- (1, 1)
  -- (2, 2) -- (3, 3) -- (4, 4) -- (5, 5) -- (6, 6) -- (7, 5) -- (8,
  4) -- (9, 5) -- (10, 4) -- (11, 3) -- (12, 2) -- (13, 1) -- (14, 0)
  -- (15, 1) -- (16, 2) -- (17, 1) -- (18, 0);

  \draw[dotted, very thick] (2, 0) -- (2, 6); \draw[dotted, very
    thick] (4, 0) -- (4, 6); \draw[dotted, very thick] (6, 0) -- (6,
  6); \draw[dotted, very thick] (8, 0) -- (8, 6); \draw[dotted, very
    thick] (10, 0) -- (10, 6); \draw[dotted, very thick] (12, 0) --
  (12, 6); \draw[dotted, very thick] (14, 0) -- (14, 6); \draw[dotted,
    very thick] (16, 0) -- (16, 6);
  
  \node (pond1) at (5, 5.5) {2}; \node (pond2) at (7, 5.5) {2}; \node
  (pond2) at (11, 3.5) {1};
\end{tikzpicture}
 }};
      \node (F1) at (0.18,-2.7) {
        \scalebox{0.35}{
          \begin{tikzpicture}
  \draw[dotted, thick, color=gray!60] (0, -1) grid (19, 6);
  \draw[rounded corners=1, color=black, line width=1] (1, 1)
  -- (2, 2) -- (3, 3) -- (4, 4) -- (5, 5) -- (6, 6) -- (7, 5) -- (8,
  4) -- (9, 5) -- (10, 4) -- (11, 3) -- (12, 2) -- (13, 1) -- (14, 0)
  -- (15, 1) -- (16, 2) -- (17, 1) -- (18, 0) -- (19, -1);

  \draw[dotted, very thick] (2, -1) -- (2, 6);
  \draw[dotted, very thick] (4, -1) -- (4, 6);
  \draw[dotted, very thick] (6, -1) -- (6, 6);
  \draw[dotted, very thick] (8, -1) -- (8, 6);
  \draw[dotted, very thick] (10, -1) -- (10, 6);
  \draw[dotted, very thick] (12, -1) -- (12, 6);
  \draw[dotted, very thick] (14, -1) -- (14, 6);
  \draw[dotted, very thick] (16, -1) -- (16, 6);
  
  \node (pond1) at (5, 5.5) {2}; \node (pond2) at (7, 5.5) {2}; \node
  (pond2) at (11, 3.5) {1};
\end{tikzpicture}
 }};
      \node (F1) at (6.3,-1.4) {
        \scalebox{1}{
          \begin{tikzpicture}[scale=1]
  \node(N5) at (0,0){5};
  \node(N2) at (-1.5,-1){2};
  \node(N8) at (1.5,-1){8};
  \node(N1) at (-2,-2){1};
  \node(N3) at (-1,-2){3};
  \node(N7) at (1,-2){7};
  \node(N9) at (2,-2){9};
  \node(N4) at (-0.5,-3){4};
  \node(N6) at (0.5,-3){6};
  
  \draw (N5) -- (N2);
  \draw (N5) -- (N8);
  \draw (N1) -- (N2);
  \draw (N3) -- (N2);
  \draw (N3) -- (N4);
  \draw (N7) -- (N8);
  \draw (N8) -- (N9);
  \draw (N6) -- (N7);
\end{tikzpicture}
 }};
    \end{tikzpicture}
    \caption{A weighted Dyck path with its intermediate path and
      its binary search tree of type~$0$.} 
    \label{treeDyck}
  \end{center}
\end{figure}
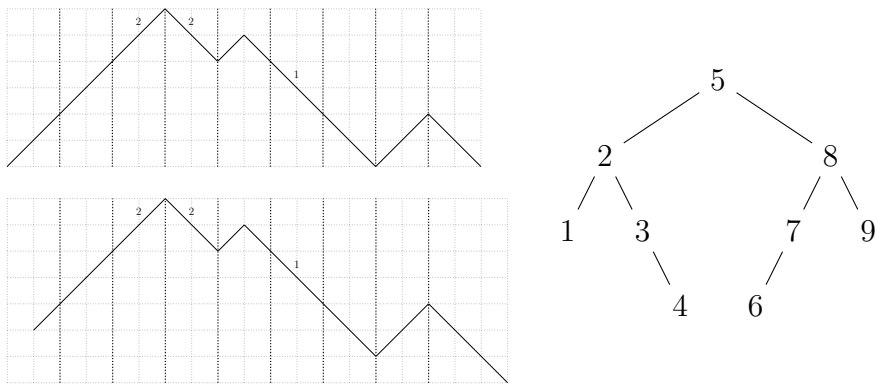

Similarly to Lemma~\ref{equalTrees}, we have the following lemma.

\begin{lemma}
  \label{equalTrees0}
  Let $\sigma$ be a permutation. Then,
  $\BST(\sigma) = \BST^0(\psi_{FV}^0(\sigma))$.
\end{lemma}

The proof of this lemma is the same as the proof of
Lemma~\ref{equalTrees} as $\psi_{FV}^0$ corresponds to the same
algorithm as $\psi_{FV}$ (with the different convention) if we remove
the first step of the resulting path and add a decreasing step at the
end.

The proof of Theorem~\ref{thmSylvester} follows directly.

\begin{proof}[Proof of Theorem \ref{thmSylvester}]
  As it can be seen in Equation~\eqref{reformPsi} in the proof of
  Lemma~\ref{lemmaStatsInvol}, we have $\psi(D,w)_{2i+1} = D_{2i}$ so
  for any weighted Dyck path $(D,w)$, we have
  $\BST(D,w)=\BST^0(\psi(D,w))$. Then, from 
  Lemma~\ref{equalTrees} and Lemma~\ref{equalTrees0} we have
  $\BST(\sigma)=\BST(\Psi(\sigma))$.
\end{proof}

\begin{remark}
  The second fundamental transformation of Foata~\cite{Lot} preserves
  the sylvester class of a permutation which is the same as the mirror
  image of its co-sylvester class. Our bijection is really different
  from Foata's and in particular they are not conjugate to each
  other. Indeed, they do not have the same number of fixed points
  over $\SG_n$ for general $n$. For example, with $n=5$, the second
  fundamental transformation of Foata has $26$ fixed points whereas
  ours has $32$. The two bijections also differ on the number of
  orbits under the action implied by the successive application of the
  map.
\end{remark}

\bibliographystyle{plain}
\bibliography{statsEquivalence}

\end{document}